\documentclass[11pt, reqno,amsppt]{article}
\usepackage[english]{babel}
\usepackage[utf8]{inputenc}
\usepackage{authblk}
\usepackage{latexsym,amsfonts,amssymb,amsthm,mathrsfs,amsmath,amscd,enumerate, enumitem,esint}
\usepackage{color}
\allowdisplaybreaks
\usepackage[pagewise]{lineno}
\newtheorem{teo}{Theorem}[section]

\newtheorem{lema}[teo]{Lemma}
\newtheorem{propo}[teo]{Proposition}
\theoremstyle{definition}
\newtheorem{definition}[teo]{Definition}
\theoremstyle{remark}
\newtheorem{remark}[teo]{Remark}

\newtheorem{ej}[teo]{Example}

\theoremstyle{estilolista}
\newtheorem{itemnumerado}[equation]{}

\usepackage{titlesec}
\numberwithin{equation}{section}

		\def\zLL{\ensuremath{\mathcal{L}}}
		
		\def\zR{\ensuremath{\mathbb{R}}}
		\def\z{\ensuremath{\mathbb{R}^n}}
		\def\zZ{\ensuremath{\mathbb{Z}}}
		\def\zN{\ensuremath{\mathbb{N}}}
		\def\zL{\ensuremath{\mathbb{L}}}
		
		\newcommand{\ff}[3]{#1:#2\rightarrow #3}
		
		\newcommand{\fx}{f(x)}
		
		\newcommand{\lpw}[2]{L^{#1(\cdot)}_#2}
		\newcommand{\lpwr}[2]{L^{#1(\cdot)}_#2(\z)}
		\newcommand{\lp}[1]{L^{#1(\cdot)}}
		\newcommand{\lploc}[1]{\lp{#1}_{\rm{loc}}(\z)}
		\newcommand{\lplocc}[1]{L^{#1}_{\rm{loc}}(\z)}
		\newcommand{\lpr}[1]{L^{#1(\cdot)}(\z)}
		\newcommand{\normadeflpw}[3]{\left\Vert#1\right\Vert_{L^{#2(\cdot)}_#3}}
		\newcommand{\normadeflpl}[2]{\left\Vert#1\right\Vert_{#2(\cdot,L)}}
		\newcommand{\normadeflp}[2]{\left\Vert#1\right\Vert_{L^{#2(\cdot)}}}
		\newcommand{\normadefpv}[2]{\left\Vert#1\right\Vert_{#2(\cdot)}}
		
		\newcommand{\normadefp}[2]{\left\Vert#1\right\Vert_{#2}}
		
		\newcommand{\esp}{\,\,\,\,\,\,\,\,\,}
		\newcommand{\texto}[1]{\textrm{#1}}
		\newcommand{\ex}[1]{#1(\cdot)}
		\newcommand{\ca}{\mathcal{X}_Q}
		
		\newcommand{\car}{\mathcal{X}}
		
		\newcommand{\inplogdern}{\in\mathcal{P}^{log}(\z)}

		\newcommand{\mbetas}[2]{M_{#1(\cdot),#2}}

\usepackage[colorlinks,linkcolor=blue,citecolor=blue]{hyperref}
\begin{document}
	\title{Musielak Orlicz bumps and Bloom type estimates for commutators of Calder\'on Zygmund and fractional integral operators on variable Lebesgue spaces via sparse operators}
	
	\author{Luciana Melchiori\thanks{lmelchiori@santafe-conicet.gov.ar, luchimelchiori@gmail.com, CONICET-UNL, Santa Fe, Argentina.}, Gladis Pradolini\thanks{gladis.pradolini@gmail.com. CONICET-UNL, Santa Fe, Argentina.}\, and Wilfredo Ramos\thanks{oderfliw769@gmail.com. CONICET-(FaCENA-UNNE), Corrientes, Argentina.}}
	
	\renewcommand{\thefootnote}{\fnsymbol{footnote}}
	\date{\vspace{-1.5cm}}
	
	\maketitle
	
	\begin{abstract}
		We obtain Musielak Orlicz bumps conditions on a pair of weights for the boundedness of Calderón Zygmund operators and their commutators between variable Lebesgue spaces with different weights.
		The symbols of the commutators belong to a wider class of functions.
		
		We also give Bloom type estimates for commutators of Calderón Zygmund and fractional integral operators in the variable Lebesgue context.
		
		The techniques involved in both type of results are related with the theory of sparse domination.
		
	\end{abstract}

\section{Introduction and main results}	

	One of the main purpose of this paper is to obtain sufficient conditions on a pair of weights in order to attain two-weighted norm inequalities for Calderón Zygmund operators (CZO's), and their commutators, between variable Lebesgue spaces. 
	We give Musielak Orlicz bump conditions on the weights that guarantee these results. 
	The symbols of the commutators belong to a wider class of functions including BMO and Lipschitz spaces.
	
	The main motivation for studying the results above is \cite{P}. 
	In this article the author studied sufficient conditions on a pair of weights in order to obtain boundedness results for potential operators between Lebesgue spaces with different weights.
	Later in \cite{CP}, a similar problem was studied for CZO's and their commutators with BMO symbols, obtaining Orlicz bump inequalities on a pair of weights as sufficient conditions. 
	In that paper, Cruz Uribe and Pérez conjectured that weaker conditions that involve Young functions are sufficient to obtain the desired boundedness. 
	This conjecture have been studied extensively, for a complete history we refer the reader to \cite{CMPc,CMPb,CMPd,L4} and \cite{CRV} for the extensive references that they contain.
	One of our result extend the main theorem in \cite{CP} to the context of Musielak Orlicz spaces. 
\medskip
	
	Another goal in this paper is to obtain Bloom type estimates on variable Lebesgue spaces for commutators of CZO's and fractional integral operators with symbols belonging to other modified Lipschitz class. 
	
	In \cite{Bloom}, Bloom obtained boundedness results of the type $L^p(\mu)\rightarrow L^p(\lambda)$ with $\mu$ and $\lambda\in A_p$, for commutator of the Hilbert transform. 
	The symbol involved belongs to a weighted version of the bounded mean oscilation space, $BMO_\nu$, where $\nu=(\mu/\lambda)^{1/p}$.
	Later, in \cite{HLW} and \cite{LOR}, the authors extend the results above to $\omega$-Calderón Zygmund operators, with $\omega(t)=t^\gamma$, $\gamma>0$, and for general $\omega$, respectively (see also \cite{LORR} for higher order commutators).
	
	On the other hand, in \cite{HRS} and \cite{AMPRR}, a version of the Bloom's result for the fractional integral operator and their commutators were given.	
	
	The principal tools in order to obtain the mentioned results are related with the sparse domination techniques (see section \S\ref{sparse}).
	\medskip
	
	We now introduce the general context where we shall be working with.
	
	Let $\ff{\ex{p}}{\z}{[1,\infty]}$ be a measurable function. For $\mathrm{A}\subset\z$ we define
	\vspace{-0.2cm}
	
	$$p^-_\mathrm{A}= ess \inf_{x\in \mathrm{A}}p(x)\esp\esp\esp p^+_\mathrm{A}= ess \sup_{x\in \mathrm{A}}p(x).$$
	For simplicity we denote $p^-=p^-_{\z}$ and $p^+=p^+_{\z}$.
	
	With $\ex{p'}$ we denote the conjugate exponent of $\ex{p}$ given by $\ex{p'}=\ex{p}/(\ex{p}-1)$.
	It is not hard to prove that $(p')^-=(p^+)'$ and $(p')^+=(p^-)'$.
	
	We say that $\ex{p}\in\mathcal{P}(\z)$ if $1\le p^-\leq p^+\le\infty$ and we denote by $\mathcal{P}^{log}(\z)$ the set of the exponents $\ex{p}\in\mathcal{P}(\z)$ that satisfy the following inequalities
	\vspace{-0.2cm}
		
	\begin{equation*}\label{log1}
			\left|\frac{1}{p(x)}- \frac{1}{p(y)}\right|\leq \frac{C}{\log(e+1/|x-y|)}, \,\,x,y\in\z
		\end{equation*}
	and
		\begin{equation}\label{log2}
			\left|\frac{1}{p(x)}- \frac{1}{p_\infty}\right|\leq \frac{C}{\log(e+|x|)}, \,\, x\in\z
		\end{equation}
	for some positive constants $C$ and $p_\infty$.
	It is easy to see that the inequality (\ref{log2}) implies that $\lim_{|x|\rightarrow \infty} {1}/{p(x)} = {1}/{p_\infty}$.
	The conditions on $1/\ex{p}$ above are known as local and global log-H\"older conditions, respectively.
	
	If $\ex{p}\in\mathcal{P}(\z)$, we define the function
	\vspace{-0.2cm}
	
			$$\varphi_{\ex{p}}(y,t)=
		\left\{
		\begin{array}{lc}
		t^{p(y)}, & 1\le p(y) <\infty \\
		\infty\cdot\car_{(1,\infty)}(t), & p(y)=\infty,
		\end{array}
		\right.$$
				
	for $t\ge0$ and $y\in\z$, with the convention $\infty\cdot 0=0$. 
	Then the variable exponent Lebesgue space $\lpr{p}$ is the set of the measurable functions $f$ defined on $\z$ such that, for some positive $\lambda$,
	\vspace{-0.9cm}
	
		$$\int_{\z} \varphi_{\ex{p}}(x,|\fx|/\lambda)\, dx<\infty.$$
	A Luxemburg norm can be defined in $\lpr{p}$ by taking
	\vspace{-0.2cm}
	
		$$\normadefpv{f}{p}=\inf\left\lbrace \lambda >0 : \int_{\z} \varphi_{\ex{p}}(x,|\fx|/\lambda)\, dx\leq 1 \right\rbrace.$$
	By $\lploc{p}$ we denote the space of the functions $f$ such that $f\in\lpr{p}(U)$ for every compact set $U\subset\z$.

	A locally integrable function $w$ defined in $\z$ which is positive almost everywhere is called a weight.
	For $\ex{p}\in\mathcal{P}(\z)$ we define the weighted variable Lebesgue space $\lpw{p}{w}(\z)$ as the set of the measurable functions $f$ defined on $\z$ such that $fw\in\lp{p}(\z)$.
	(See \cite{CF2} and \cite{DHHR} for more information about varible Lebesgue spaces).
\medskip

	By a cube $Q$ in $\z$ we shall understand a cube with sides parallel to the coordinate axes. By $\car_Q$ and $f_Q$ we denote the characteristic function of $Q$ and the average of $f$ over $Q$, respectively.

	We shall say that $A\lesssim B$ if there exist a positive constant $C$ such that $A\leq C B$.

	Throughout this paper, we use $m$ to denote a nonnegative integer. 

	We now introduce the operators we shall be working with and state the corresponding main results for each one. 
	
	Let $\omega : [0, 1] \rightarrow [0,\infty)$ a continuous, increasing and subadditive function such that $\omega(0) = 0$. 
	We say that a linear operator $T$ is an $\omega$-Calderón-Zygmund operator on $\z$ if $T$ is bounded on $L^2(\z)$, and can be represented as
	\vspace{-0.2cm}
	
		$$Tf(x) = \int_{\z}	K(x, y)f(y)\,dy,\esp x\notin \text{supp}\, f.$$
	The kernel $K$ satisfyies the size condition 
	\vspace{-0.2cm}
	
		$$|K(x, y)| \leq \frac{C_K}{|x-y|^n},\esp x\neq y,$$		
	for some positive constant $C_K$, and the smoothness condition given by
	\vspace{-0.2cm}
	
		$$|K(x, y)-K(x', y)|+|K(y, x)-K(y, x')|\leq \omega\left(\frac{|x-x'|}{|x-y|}\right)\frac{1}{|x-y|^{n}},\esp|x-y|>2|x-x'|.$$ 
	
	We denote $T\in \omega$-CZO if $T$ is an $\omega$-Calderón-Zygmund operator with $\omega$ satisfiying the Dini condition 
	\vspace{-0.3cm}
	
		$$\int_{0}^{1}\omega(t)\,\frac{dt}{t}< \infty.$$ 
	
	Given a linear operator $T$ and a locally integrable function $b$, formally define the commutator of $T$ with symbol $b$ by 
	\vspace{-0.2cm}
	
		\begin{equation*}\label{commutators}
		[b,T]f(x)=b(x)\,Tf(x)-T(bf)(x),\esp x\in\z.
		\end{equation*}
	
	The higher order commutator of order $m$ of $T$ is defined by 
	\vspace{-0.2cm}
	
		\begin{equation*} 
		T_{b}^0=T, \quad T_{b}^m=\left[b, T_b^{m-1}\right]. 
		\end{equation*} 
	
	We say that the functional $a$ satisfies the $T_\infty$ condition (and we denote $a\in T_\infty$) if there exists a positive constant $\mathbf{t}_\infty$ such that for every cube $Q$ and every cube $Q'\subset Q$,
		\begin{equation}\label{cond t infinito}
			a(Q') \leq \mathbf{t}_\infty a(Q).
		\end{equation}
	We denote the least constant $\mathbf{t}_\infty$ in (\ref{cond t infinito}) by $\|a\|_{\mathbf{t}_\infty}$. 
	Clearly, $\|a\|_{\mathbf{t}_\infty}\ge1$, 
	
	Let $a\in T_\infty$, we say that a function $b\in\lplocc{1}$ belongs to the generalized Lipschitz space $\mathcal{L}_a$ if
	\vspace{-0.2cm}
	
		\begin{equation*}\label{lipschitz a}
			\sup_Q \frac{1}{a(Q)|Q|}\int_Q |b(x)-b_Q|\,dx<\infty,
		\end{equation*}
	where the supremum is taken over all cubes $Q\subset\z$.

	We are now in position to state our first result.
	\begin{teo}\label{conmutador simbolo lip general}
		Let $T\in \omega$-CZO and $\ex{p}\in\mathcal{P}^{log}(\z)$ such that $1<p^-\le p^+<\infty$.
		Let $m\in\zN\cup\{0\}$  and $b\in\mathcal{L}_a$ with $a\in T_\infty$.  
		Suppose that $(v,w)$ is any couple of weights such that $v\in \lploc{p}$ and for some constants $S>p^+/p^-$ and $R>(p')^+/(p')^-$,
			\begin{equation}\label{cond pesos teo CZ lipschitz general}
				\sup_Q \,a(Q)^{m}
				\frac{\normadefpv{\car_{Q}w}{Sp}}{\normadefpv{\car_{Q}}{Sp}}
				\frac{\normadefpv{\car_{Q}v^{-1}}{Rp'}}{\normadefpv{\car_{Q}}{Rp'}}
				<\infty.
			\end{equation}
		Then 
		\vspace{-0.3cm}
			
			$$T^m_b:\lpw{p}{v}(\z)\hookrightarrow\lpw{p}{w}(\z).$$
	\end{teo}
\medskip
	\begin{remark}
		Note that the weights $w$ and $v=\sup_{Q}a(Q)^{m}{\normadefpv{\car_Qw}{Sp}}/{\normadefpv{\car_Q}{Sp}}$ where $S>p^+/p^-$ satisfies condition \eqref{cond pesos teo CZ lipschitz general}.
		In fact,
			\begin{eqnarray*}
				a(Q)^{m}
				\frac{\normadefpv{\car_Qw}{Sp}}{\normadefpv{\car_Q}{Sp}}
				\frac{\normadefpv{\car_Qv^{-1}}{Rp'}}{\normadefpv{\car_Q}{Rp'}}
				\leq				
				\frac{\normadefpv{\car_Qw}{Sp}}{\normadefpv{\car_Q}{Sp}}
				\frac{\normadefpv{\car_Q}{Sp}}{\normadefpv{\car_Qw}{Sp}}
				=1.
			\end{eqnarray*}
	\end{remark}
\medskip

	In the classical Lebesgue spaces, a proof can be found in \cite{CP} for the case $\omega(t)=t^\gamma$, $\gamma>0$ and $b\in BMO=\mathcal{L}_a$ with $a\equiv1$.
	
	Let us observe that, if $a(Q) = |Q|^{\delta/n}$, $0 < \delta < 1$, then $a\in T_\infty$ and it is known that $\mathcal{L}_a:=\zL(\delta)$ coincides with the classical Lipschitz spaces define as the set of functions $b$ such that
	\vspace{-0.4cm}
	
		$$|b(x)-b(y)|\lesssim |x-y|^\delta,\esp x,y\in\z.$$
	
	On the other hand, if $\ex{r}\in\mathcal{P}^{log}(\z)$, $n/r^-\le \alpha $ and $\ex{\delta}$ is the exponent defined by 
		\begin{equation}\label{delta}
			\frac{\ex{\delta}}{n}=\frac{\alpha}{n}-\frac{1}{\ex{r}},
		\end{equation}
	the functional $a(Q)=\normadefpv{\car_Q}{n/\delta}$ satisfies the $T_\infty$ condition and $\mathcal{L}_a=\zL(\ex{\delta})$ is a variable version of the spaces $\zL(\delta)$ defined above. 

	Particularly, if $b\in\zL(\ex{\delta})$, we can improve the theorem above in the sense that we can consider weaker norms on the weights than those in \eqref{cond pesos teo CZ lipschitz general}, involving generalized $\Phi$-functions, denoted by G$\Phi$-functions (see section \S\ref{GFF} for more information about G$\Phi$-functions).
	In order to state the result we need some previous definitions.
	
	The norm associated to a given G$\Phi$-function $\Psi$ is define by
	\vspace{-0.2cm}
	
	$$\normadeflpl{f}{\Psi}=\inf\left\{\lambda>0:\int_{\z}\Psi\left(x,\frac{|f(x)|}{\lambda}\right)\,dx\leq 1\right\}$$
	and we denote by $L^\Psi(\z)$ the space of functions $f$ such that $\normadeflpl{f}{\Psi}<\infty$.
	
	A corresponding maximal operator associated to $\Psi$ is
	\vspace{-0.2cm}
	
		\begin{equation*}
			M_{\Psi(\cdot,L)} f(x)=\sup_{Q\ni x}\frac{\normadeflpl{\ca f}{\Psi}}{\normadeflpl{\ca}{\Psi}},\esp x\in\z.
		\end{equation*}
	For $\ex{\beta}\in\mathcal{P}(\z)$, a fractional type version of the maximal defined above is given by 
	\vspace{-0.2cm}
	
		\begin{equation*}\label{max psi}
			M_{\ex{\beta},\Psi(\cdot,L)} f(x)=\sup_{Q\ni x}\normadefpv{\ca}{\beta}\frac{\normadeflpl{\ca f}{\Psi}}{\normadeflpl{\ca}{\Psi}},\esp x\in\z.
		\end{equation*}
	
	We say that a 3-tuples of G$\Phi$-functions $(\mathsf{A},\mathsf{B},\mathsf{D})$ satisfy condition $\mathcal{F}$ if they verify
	
	\begin{itemnumerado}\label{desigualdad normas}
		$\normadefp{\ca}{\mathsf{A}(\cdot,L)}\normadefp{\ca}{\mathsf{B}(\cdot,L)}\lesssim\normadefp{\ca}{\mathsf{D}(\cdot,L)}$.
	\end{itemnumerado}
	\begin{itemnumerado}\label{a b y c para holder}
		$\mathsf{A}^{-1}(x,t)\mathsf{B}^{-1}(x,t)\lesssim \mathsf{D}^{-1}(x,t)$ where $\mathsf{A}^{-1}$ denotes the inverse of $\mathsf{A}$ (for the definition of the inverse of a G$\Phi$-function see section \S\ref{GFF}).	\end{itemnumerado}	\begin{itemnumerado}
		$\normadeflpl{\ca}{\mathsf{D}}\normadeflpl{\ca}{\mathsf{D}^*}\lesssim |Q|$, where $\mathsf{D}^*$ is the conjugate function of $\mathsf{D}$ (for its definition see section \S\ref{GFF}).\label{D y D conjugada}
	\end{itemnumerado}
	
	Necessary conditions on $\mathsf{D}$ where given in \cite{DHHR} 
	in order to verify \textbf{\ref{D y D conjugada}}.
\medskip 

	We can now state our result.
	
	\begin{teo}\label{conmutador simbolo lip variable}
		Let $T\in \omega$-CZO and let $\ex{p}\in\mathcal{P}(\z)$. 
		Let $0<\alpha<n$ and $\ex{r}\in\mathcal{P}^{log}(\z)$ such that $n/\alpha<r^-$ and $r_\infty\leq\ex{r}$, $\ex{\delta}$ be defined as in (\ref{delta}) and $b\in\zL(\ex{\delta})$.
		Assume that $(\mathsf{A},\mathsf{B},\mathsf{D})$ and $(\mathsf{E},\mathsf{H},\mathsf{J})$ are 3-tuples of G$\Phi$-functions satisfying condition $\mathcal{F}$ and 
			\begin{equation}\label{hipotesis maximal}
			M_{\mathsf{B}(\cdot,L)}:L^{\ex{p}}(\z)\hookrightarrow L^{\ex{p}}(\z) \esp \esp
			M_{\mathsf{H}(\cdot,L)}:L^{\ex{p'}}(\z)\hookrightarrow L^{\ex{p'}}(\z).
			\end{equation}
		Suppose that $(v,w)$ is any couple of weights such that $v\in \lploc{p}$ and
			\begin{equation}\label{cond 1}
				\sup_Q\,\normadefpv{\car_Q}{n/\delta}^m
				\frac{\normadefp{\car_Qw}{\mathsf{E}(\cdot,L)}}
				{\normadefp{\car_Q}{\mathsf{E}(\cdot,L)}}
				\frac{\normadefp{\car_Qv^{-1}}{\mathsf{A}(\cdot,L)}}
				{\normadefp{\car_Q}{\mathsf{A}(\cdot,L)}}
				<\infty.
			\end{equation}
		Then 
		\vspace{-0.3cm}
		
			$$T^{m}_b:\lpw{p}{v}(\z)\hookrightarrow\lpw{p}{w}(\z).$$
	\end{teo}

 	Let us give some examples of G$\Phi$-functions that satisfy the hypothesis of the theorem above.
 	Notice first that, if we consider $\ex{p}\in\mathcal{P}(\z)$ and $\ex{q}$ with $q^+<\infty$, then for $x\in\z,$ $t\ge0$,	$\Psi(x,t)=t^{p(x)}(\log(e+t))^{q(x)}$	is a G$\Phi$-function.
	 	In this case, the space $L^\Psi(\z)$ will be denoted by $L^{\ex{p}}(\log L)^{\ex{q}}(\z)$.
	 	In [\cite{MMO}, Proposition 2.5] the authors proved that the Hardy-Littlewood maximal operator $M$ is bounded in this space when $\ex{p}\inplogdern$ with $1<p^-\le p^+<\infty$, and $\ex{q}\in\mathcal{P}^{{\rm loglog}}(\z)$.	
	 	We say that $\ex{q}\in\mathcal{P}^{{\rm loglog}}(\z)$ if $\ff{\ex{q}}{\z}{\zR}$ with $q^+<\infty$ such that, for some positive constant $C$, it satisfies the following inequality
	 	\vspace{-0.2cm}
	 	
			\begin{equation*}\label{loglog}
			|q(x)- q(y)|\leq \frac{C}{\log(e+\log(e+1/|x-y|))}, \textrm{ for every } x,y\in\z.
			\end{equation*}
	
	Let $\ex{p}\inplogdern$ with $1<p^-\le p^+<\infty$ and $\sigma>(p')^+/(p')^-$.
	The following G$\Phi$-functions satisfy condition $\mathcal{F}$ and the hyphoteses \eqref{hipotesis maximal} of the Theorem \ref{conmutador simbolo lip variable}.	
	
	\begin{ej}\label{ejemplo 1}
		$\mathsf{A}_1(x,t)=t^{\sigma p'(x)}(\log(e+t))^{\sigma p'(x)}$,
		$\mathsf{B}_1(x,t)=t^{(\sigma p')'(x)}$ and
		$\mathsf{D}_1(t)=t \log(e+t)$.
	\end{ej}

	\begin{ej}\label{ejemplo 2}
		If, in addition, $\ex{\mu}\inplogdern$ with $1<\mu^-\le\mu^+<\infty$ such that 
		\vspace{-0.2cm}
		
		\begin{equation*}
		1/\sigma\ex{p'}-1/\ex{\mu}>\epsilon
		\end{equation*}
		
		for some constant $\epsilon\in(0,1)$ and $\ex{\nu}\in\mathcal{P}^{{\rm loglog}}(\z)$ then, an example is given by
		\vspace{-0.2cm}
		
		$$\mathsf{A}_2(x,t)=t^{\mu(x)}(\log(e+t))^{\nu(x)\mu(x)},\, \mathsf{B}_2(x,t)=t^{(\sigma p')'(x)}$$ and $\mathsf{D}_2(x,t)=t^{\alpha(x)} (\log(e+t))^{\alpha(x)\nu(x)}$ where $\ex{\alpha}$ is defined by $1/\ex{\alpha}=1/\ex{\mu}+1/\ex{(\sigma p')'}.$
		
		In \cite{MPR} we checked above examples.
	\end{ej}

	\begin{remark}
		Note that the pair of weights $(w,v)$, where $v$ is defined by 
		\vspace{-0.2cm}
		$$v(x)=\sup_{Q\ni x}\normadefpv{\ca}{n/\delta}^{m}{\normadefp{\ca w}{\mathsf{E}(\cdot,L)}}/{\normadefp{\ca}{\mathsf{E}(\cdot,L)}}$$ satisfies condition \eqref{cond 1}.
		In fact,
		\vspace{-0.2cm}
		
		\begin{eqnarray*}
			\normadefpv{\ca}{n/\delta}^{m}
			\frac{\normadefp{\ca w}{\mathsf{E}(\cdot,L)}}{\normadefp{\ca}{\mathsf{E}(\cdot,L)}}
			\frac{\normadefp{\ca v^{-1}}{\mathsf{A}(\cdot,L)}}{\normadefp{\ca}{\mathsf{A}(\cdot,L)}}
			\leq				
			\frac{\normadefp{\ca w}{\mathsf{E}(\cdot,L)}}{\normadefp{\ca}{\mathsf{E}(\cdot,L)}}
			\frac{\normadefp{\ca}{\mathsf{E}(\cdot,L)}}{\normadefp{\ca w}{\mathsf{E}(\cdot,L)}}
			=1.
		\end{eqnarray*}
	\end{remark}
	
\bigskip

	Another class of symbols we shall consider is related with the Bloom type estimates in the variable Lebesgue spaces.
	
		\begin{definition}
			Let $\eta$ be a weight, $0\le\ex{\delta}<n$ and $b\in\lplocc{1}$. 
			We say that $b\in BMO_\eta^{\ex{\delta}}$ if 
			\vspace{-0.2cm}
			
			$$\normadefp{b}{BMO_\eta^{\ex{\delta}}}=\sup_{Q}\frac{1}{\normadefpv{\ca}{n/\delta}\eta(Q)}\int_Q|b-b_Q|<\infty.$$
		\end{definition}
	
	When $\ex{\delta}\equiv0$, $BMO^0_\nu=BMO_\nu$, the space introduced in \cite{Bloom}. 
	If $\ex{\delta}\equiv\delta$, with $0<\delta<1$, $BMO^\delta_\nu$ is a Lipschitz type space defined in \cite{HSV}.
	
	Our first result generalizyng Bloom's theorem in the variable Lebesgue context for CZO is the following.
	For the definitions of the classes of weights see section \S\ref{some previous}.
	
		\begin{teo}\label{teo simb lips bloom}
			Let $T\in \omega$-CZO and	
			$\ex{p}, \ex{q}\inplogdern$ 
			such that $1<p^-\le\ex{p}<\ex{q}\le q^+<\infty$. 
			Let $\ex{\delta}$ be the exponent defined by 
			\vspace{-0.2cm}
			
				$$\frac{m\ex{\delta}}{n}= \frac{1}{\ex{p}}-\frac{1}{\ex{q}}.$$
			Let $\mu,\lambda\in A_{\ex{p},\ex{q}}$ and $\nu={\mu}/{\lambda}$. 
			Assume that $b\in BMO_{\nu^{1/m}}^{\ex{\delta}}$, then
			\vspace{-0.2cm}
			
				$$T^m_b:\lpw{p}{\mu}(\z)\hookrightarrow\lpw{q}{\lambda}(\z).$$
		\end{teo}
	
	When $p=q$ are constants, Theorem \ref{teo simb lips bloom} was proved in \cite{LOR} for the first order commutator and in \cite{LORR} for higher order.
	
	A similar result in the spirit of theorem above for the higher order commutator of the fractional integral operator is given by the following theorem.
	Recall first that the fractional integral operator is defined, for $0<\alpha<n$, by 
	\vspace{-0.2cm}
	
		$$I_\alpha f(x)=\int_{\z}\frac{f(y)}{|x-y|^{n-\alpha}}\,dy, \esp x\in\z.$$
	
	\begin{teo}\label{conmutador integral fraccionaria en lpv con simb bloom}
		Let $0<\alpha<n$ and $\ex{p}, \ex{q}\inplogdern$ such that $1<p^-\le\ex{p}<\ex{q}\le q^+<\infty$. 
		Let $\ex{\delta}$ be the exponent defined by 
		\vspace{-0.2cm}
		
			$$\frac{m\ex{\delta}+\alpha}{n}= \frac{1}{\ex{p}}-\frac{1}{\ex{q}}.$$ 
		Let $\mu,\lambda\in A_{\ex{p},\ex{q}}$ and $\nu={\mu}/{\lambda}$. 
		Assume that $b\in BMO^{\ex{\delta}}_{\nu^{1/m}}$, then 
		\vspace{-0.2cm}
		
			$$(I_\alpha)^m_b:\lpw{p}{\mu}(\z)\hookrightarrow\lpw{q}{\lambda}(\z).$$
	\end{teo}
	
	When $p,q$ are constants and $1/{p}-1/{q}=\alpha /n$, the result above was proved in \cite{HRS} for the first order commutator and in \cite{AMPRR} for higher order.

\section{Preliminaries}
	In order to prove our results we give some preliminaries definitions and technical lemmas.
	
\subsection{Sparse operators}\label{sparse}
	We now introduce the dyadic structures we will working with. These definitions and a profound treatise on dyadic calculus can be
	found in \cite{LN}.
	
	We say that a collection of cubes $\mathscr{D}$ in $\z$ is a dyadic grid if it satisfies the following properties:
	\begin{enumerate}
		\item If $Q\in \mathscr{D}$, then $\ell(Q) = 2^k$ for some $k\in\zZ$.
		\item If $P,Q\in\mathscr{D}$, then $P\cap Q\in\{P,Q,\emptyset\}$.
		\item For every $k\in\zZ$, the cubes $\mathscr{D}_k =\{Q \in \mathscr{D} : \ell(Q) = 2^k\}$ form a partition of $\z$.	
	\end{enumerate}
	
	Given a dyadic grid $\mathscr{D}$, a set $\mathcal{S}\subset \mathscr{D}$ is sparse if for every $Q\in \mathcal{S}$,
	\vspace{-0.2cm}
	
	$$\left|\bigcup_{\overset{P\in \mathcal{S}}{P\subsetneq Q}}P\right|\leq \frac{1}{2} |Q|.$$
	Equivalently, if we define
	\begin{equation}\label{cubos E(Q)}
	E(Q)=Q\setminus \bigcup_{\overset{P\in \mathcal{S}}{P\subsetneq Q}}P,
	\end{equation}
	then the sets $E(Q)$ are pairwise disjoint and $|E(Q)|\geq\frac{ 1}{2}|Q|$.
	
	The classic example of a dyadic grid and sparse family are the standard dyadic grid on $\z$ and the Calderón-Zygmund cubes associated with an $L^1$ function.
	
	
	\bigskip
	The following results establish pointwise sparse domination for higher order commutators of $T\in \omega$-CZO and the fractional integral operator $I_\alpha$.
	For simplicity we introduce the following notation.
	Let $m,h$ be two integers, $0\le\alpha<n$ and $\mathcal{S}$ be a sparse family, we denote $\mathcal{A}^{m,h}_{\mathcal{S},\alpha}$ the fractional sparse operator given by	
	\vspace{-0.4cm}
	
	\begin{align*}
	\mathcal{A}^{m,h}_{\mathcal{S},\alpha}(b,f)(x)
	&=\sum_{Q\in \mathcal{S}}
	|b(x)-b_{Q}|^{m-h}
	|Q|^{\alpha/n}	
	|(b-b_{Q})^{h}f|_{Q}\cdot\ca(x), \esp b,f\in\lplocc{1}.
	\end{align*}
	When $\alpha=0$ we denote $\mathcal{A}^{m,h}_{\mathcal{S},0}=\mathcal{A}^{m,h}_{\mathcal{S}}$ and, if $m=h=\alpha=0$ we denote $\mathcal{A}^{0,0}_{\mathcal{S},0}(b,f)=\mathcal{A}_{\mathcal{S}}(f)$.
	
	\begin{teo}[\cite{IFRR}]\label{est puntual del conmutador de orden m}
		Let $T\in \omega$-CZO. 
		For every bounded function $f$ with compact support and $b\in\lplocc{1}$, 
		there exist $3^n$ 
		sparse families $\mathcal{S}_j$ 
		such that 
		\vspace{-0.2cm}
		
		\begin{equation*}\label{desigualdad puntual}
		|T^m_bf(x)|
		\leq C(m,n,T) \sum^{3^n}_{j=1}
		\sum_{h=0}^{m} \mathcal{A}_{{\mathcal{S}_j}}^{m,h}(b,f)(x),\esp a.e.\,\,x\in\z.
		\end{equation*}
	\end{teo}
	
	A more general version of the statement above was proved in \cite{IFRR}. 	
	
	\begin{teo}[\cite{AMPRR}]\label{est puntual del conmutador de la int fracc de orden m}
		Let $0 <\alpha < n$. 
		For every bounded function $f$ with compact support and $b\in\lplocc{m}$, there exist $3^n$ sparse families $\mathcal{S}_j$ such that
		\vspace{-0.2cm}
		
		\begin{equation*}\label{desigualdad puntual int fracc}
		|(I_\alpha)^m_bf(x)|
		\leq C(m,n,T) \sum^{3^n}_{j=1}\sum_{h=0}^{m} \mathcal{A}_{{\mathcal{S}_j,\alpha}}^{m,h}(b,f)(x),\esp a.e.\,\,x\in\z.
		\end{equation*}
	\end{teo}
	
	We shall use the following result in the proof of Proposition \ref{integral de b por f menor q la integral del operador sparse}.
	\begin{lema}[\cite{LOR}]\label{LORR Lemma 5.1.}
		Let $\mathscr{D}$ be a dyadic grid and let $\mathcal{S}\subset\mathscr{D}$ be a sparse family. 
		Assume that $b\in\lplocc{1}$. 
		Then there exists a sparse family $\tilde{\mathcal{S}}\subset\mathscr{D}$ such that $\mathcal{S}\subset\tilde{\mathcal{S}}$ and for every cube $Q\in\tilde{\mathcal{S}}$,
		\vspace{-0.2cm}
		
		$$|b(x)-b_Q|\lesssim \sum_{\overset{R\in \tilde{\mathcal{S}}}{R\subseteq Q}}\frac{1}{|R|}\int_R|b(y)-b_R|\,dy\cdot\car_R(x),\esp a.e.\,\,x\in Q.$$
	\end{lema}
		
\subsection{Generalized $\Phi$-functions}\label{GFF}
	With $\mathcal{M}$ we denote the set of all Lebesgue real valued, measurable functions on $\z$.
	
	A convex function $\psi:[0,\infty)\rightarrow[0,\infty)$ with
	$\psi(0) = 0$,
	$\lim_{t\rightarrow 0^+}\psi(t) = 0$ and $\lim_{t\rightarrow \infty}\psi(t) =\infty$ is called a $\Phi$-function.
	
	A real function $\Psi:\z\times[0,\infty)\rightarrow[0,\infty)$ is said to be a generalized $\Phi$-function (G$\Phi$-functon), and we denote $\Psi\in\Phi(\z)$,
	if $\Psi(x,t)$ is Lebesgue-measurable in $x$ for every $t\geq0$ and $\Psi(x,\cdot)$ is a $\Phi$-function
	for every $x\in\z$.
	
	If $\Psi\in \Phi(\z)$, then the set
	\vspace{-0.2cm}
	
	$$L^\Psi(\z)=\left\{ f\in\mathcal{M}: \int_{\z} \Psi\left(x,|f(x)|\right)\,dx<\infty\right\}$$
	defines a Banach function space equipped with the Luxemburg-norm given by
	\vspace{-0.2cm}
	
	$$\normadeflpl{f}{\Psi}=\inf\left\{ \lambda>0: \int_{\z} \Psi\left(x,\frac{|f(x)|}{\lambda}\right)\,dx\leq1\right\}.$$
	The space $L^\Psi(\z)$ is called a Musielak-Orlicz space.
	
	Let $\ex{p}\in\mathcal{P}(\z)$, then $\Psi(x,t)=t^{p(x)}\in \Phi(\z)$.
	In this case, the space $L^\Psi(\z)$ is the variable exponent Lebesgue space $L^{\ex{p}}(\z)$ defined in the introduction.
	
	Let $\Psi\in \Phi(\z)$, then for any $x\in\z$ we denote by $\Psi^*(x,\cdot)$ the conjugate function of $\Psi(x,\cdot)$ which is defined by
	\vspace{-0.2cm}
	
	$$\Psi^*(x,u)=\sup_{t\geq0}\,(tu-\Psi(x,t)),\esp  u\geq 0.$$
	Also we can define $\Psi^{-1}$, the generalized inverse function of $\Psi$ by
	\vspace{-0.2cm}
	
	$$\Psi^{-1}(x,t)=\inf\{u\geq0: \Psi(x,u)\geq t\},\esp  x\in\z, t\geq 0.$$	
	
	The following result is a generalization of the classical H\"older inequality to the Musielak–Orlicz spaces.
	
	\begin{lema}
		Let $\Psi\in\Phi(\z)$, then
		\begin{equation}\label{holdermusie}
		\int_{\z}f(x)g(x)\,dx \lesssim\normadeflpl{f}{\Psi}\normadeflpl{g}{\Psi^*}
		\end{equation}
		for all $f\in L^\Psi(\z)$ and $g\in L^{\Psi^*}(\z)$.
	\end{lema}

	For the definition of $\Psi^*$, the following generalization of the Young's inequality holds in this context,
	\vspace{-0.5cm}
	
	\begin{equation*}\label{young}
	tu\leq \Psi(\omega,t)+\Psi^*(\omega,u),\esp \omega\in\z,\,\,t,u\geq0
	\end{equation*}
	for any $\Psi\in\Phi(\z)$.
	Moreover, it can be proved that if $\Psi,\Lambda,\Theta\in \Phi(\z)$ and
	$\Psi^{-1}(x,t)\Lambda^{-1}(x,t)\leq\Theta^{-1}(x,t)$
	then
	\vspace{-0.5cm}
	
	$$\Theta(x,tu)\leq \Psi(x,t)+\Lambda(x,u).$$
	
	The inequality above allow us to prove the following generalized H\"older type inequality.
	
	\begin{lema}[\cite{MP2}]
		Let $\Psi,\Lambda,\Theta\in \Phi(\z)$ such that
		$\Psi^{-1}(x,t)\Lambda^{-1}(x,t)\leq\Theta^{-1}(x,t).$
		Then
			\begin{equation}\label{holdermusie2}
				\normadeflpl{fg}{\Theta} \lesssim \normadeflpl{f}{\Psi}\normadeflpl{g}{\Lambda}
			\end{equation}
		for all $f\in L^\Psi(\z)$ and $g\in L^{\Lambda}(\z)$.
	\end{lema}
	See \cite{DHHR} and \cite{HH} for more information about generalized $\Phi$-functions.
	\medskip
\subsection{Variable Lebesgue spaces}\label{lpv}
	When we deal with variable Lebesgue spaces, we have the following known results that we shall be using along this paper.
		
	\begin{lema}[\cite{DHHR}]
		Let $\ex{s},\ex{p},\ex{q}\in\mathcal{P}(\z)$ be such that $1/\ex{s}=1/\ex{p}+1/\ex{q}$. Then
			\begin{equation}\label{holderrpq}
			\normadefpv{fg}{s}\lesssim\normadefpv{f}{p}\normadefpv{g}{q}.
			\end{equation}
		Particularly, if $\ex{s}\equiv1$, the inequality above gives
			\begin{equation}\label{holderpp}
			\int_{\z}|f(y)g(y)|\,dy\lesssim
			\normadefpv{f}{p}\normadefpv{g}{p'}
			\end{equation}
		which is an extension of the classical H\"{o}lder inequality.		
	\end{lema}

	\begin{lema}[\cite{DHHR}]\label{3.2.6}
		Let $\ex{p}\in\mathcal{P}(\z)$ and $s\geq 1/p^-$.
		Then $\normadefpv{|f|^s}{p}=\normadefpv{f}{sp}^s$.	
	\end{lema}	

	\begin{lema}[\cite{DHHR}]\label{p en plog}
		Let  $\ex{p}\in \mathcal{P}^{log}(\zR^n)$.
		Then $\normadefpv{\car_Q}{p}\normadefpv{\car_Q}{p'} \simeq |Q|,$
		for every cubes $Q\subset\z$.
	\end{lema}

	Moreover, we have the following result.
	\begin{lema}[\cite{MP2}]\label{equivalenciabeta}
		Let $\ex{p}, \ex{q}\inplogdern$ such that $\ex{p}\leq \ex{q}$.
		Suppose that $1/\ex{p}=1/\ex{\beta}+1/\ex{q}$ then, for every cube $Q\subset\z$, $\normadefpv{\car_Q}{p} \simeq \normadefpv{\car_Q}{\beta}\normadefpv{\car_Q}{q}$.
	\end{lema}

    \begin{teo}[\cite{DHHR}]\label{teo logl} Let $\ex{p},\ex{s},\ex{l}\in\mathcal{P}^{log}(\z)$ such that $\ex{p}=\ex{s}l(\cdot)$ with $l^->1$. 
    	Then \vspace{-0.2cm}
    			
		$$M_{L^{\ex{s}}}:\lp{p}(\z) \hookrightarrow \lp{p}(\z).$$
	\end{teo}

	\begin{teo}[\cite{MP2}]\label{teo max llogl}
		Let  $\ex{p},\ex{q}\in \mathcal{P}^{log}(\z)$ such that $\ex{p}\leq\ex{q}$.
		Let $\ex{\beta}$ and $\ex{s}\inplogdern$ be two functions such that $1/\beta(\cdot)=1/\ex{p}-1/\ex{q}$, $(p/s)^->1$ and $s^+<\infty$.
		Then \vspace{-0.2cm}
		
			\begin{equation*}
				\mbetas{\beta}{L^{\ex{s}}}:\lp{p}(\z) \hookrightarrow \lp{q}(\z).
			\end{equation*}
	\end{teo}

	Let $p\in\mathcal{P}(\z)$, we say that a weight $w\in A_{\ex{p}}$ if there exists a positive constant $C$ such that, for every cube $Q\subset\z$,  
	\begin{equation}\label{Ap}
	{\normadefpv{\ca w}{p}}
	{\normadefpv{\ca w^{-1}}{p'}}\leq C|Q|.
	\end{equation}
	
	\begin{lema}[\cite{CDH}]\label{s de la apertura clase Ap}
		Let $\ex{p}\inplogdern$ with $p^->1$ and $w\in A_{\ex{p}}$. Then there exist a constant  $s\in(1/p^-,1)$ such that $w^{1/s}\in A_{s\ex{p}}$. 
	\end{lema}
		
	\begin{lema}[\cite{MPR}]\label{Izuki con lipschitz general}
		Let $k$ be a positive integer and $\ex{p}\inplogdern$ with $1<p^-\le p^+<\infty$.
		Let $a\in T_\infty$ and $b\in \mathcal{L}_a$.
		Then, for every cube $Q\subset\z$,
		\vspace{-0.2cm}
		
		\begin{equation*}
		\frac{\normadefpv{\car_Q(b-b_Q)^k}{p}}{\normadefpv{\car_Q}{p}}
		\lesssim a(Q)^k\normadefp{b}{\mathcal{L}_a}^k.
		\end{equation*}
	\end{lema}
	
	\begin{lema}[\cite{MPR}]\label{obs teorema delta variable}
		Let $\ex{r}\in\mathcal{P}^{log}(\z)$ with $r_\infty\leq\ex{r}$, $\ex{\delta}$ be defined as in (\ref{delta}) and $b\in\mathbb{L}(\ex{\delta})$. Let $Q$ be a cube in $\z$ and $z\in kQ$ for some positive integer $k$.
		Then
		\vspace{-0.2cm}
		
			$$\left|b(z)-b_Q\right|
			\lesssim\normadefp{\ca}{n/\ex{\delta}}.			$$
	\end{lema}

\section{Some previous results}\label{some previous}

	The following proposition is useful in order to prove the Theorem \ref{teo simb lips bloom}. 
	\begin{propo}\label{integral de b por f menor q la integral del operador sparse}
		Let $\eta$ be a weight, $0\leq\ex{\delta}<n$ and $k$ a non negative integer.
		Let $\mathcal{S}$ be a sparse family contained in a dyadic grid $\mathscr{D}$.  
		Assume that $b\in BMO_\eta^{\ex{\delta}}$. 
		Then there exists a sparse family $\tilde{\mathcal{S}}\subset\mathscr{D}$ such that $\mathcal{S}\subset\tilde{\mathcal{S}}$ and for every cube $Q\in\tilde{\mathcal{S}}$, 
		\vspace{-0.3cm}
		
			\begin{align*}
				\int_{Q}|b(x)-b_Q|^k|f(x)|\,dx
				\lesssim \normadefp{b}{BMO^{\ex{\delta}}_\eta}^k\normadefpv{\ca}{n/\delta}^{k}
				\int_{Q} (\mathcal{A}_{\tilde{\mathcal{S}}})^k_\eta f(x)\,dx, \esp f\in \lplocc{1},
			\end{align*}
		where 
		$(\mathcal{A}_{\tilde{\mathcal{S}}})_\eta f(x)=\eta(x)\mathcal{A}_{\tilde{\mathcal{S}}}f(x)$ and $(\mathcal{A}_{\tilde{\mathcal{S}}})^k_\eta$ denotes the operator $(\mathcal{A}_{\tilde{\mathcal{S}}})_\eta$ iterated $k$ times.
	\end{propo}

	When $\ex{\delta}\equiv0$, the result above was proved in \cite{LOR} in the case $k=1$, and  in \cite{LORR} for $k>1$.
	\begin{proof}[Proof of Proposition~\ref{integral de b por f menor q la integral del operador sparse}]
		Let $\tilde{\mathcal{S}}$ be the sparse family provided by Lemma \ref{LORR Lemma 5.1.} and $Q\in\tilde{\mathcal{S}}$. Then, by this lemma, we have
		\vspace{-0.4cm}
		
			\begin{align*}
				|b(x)-b_Q|
				&\lesssim 
				\sum_{\overset{R\in \tilde{\mathcal{S}}}{R\subset Q}} \frac{1}{|R|}\int_R|b(y)-b_R|\,dy\cdot\car_{R}(x)
				\lesssim \normadefp{b}{BMO^{\ex{\delta}}_\eta}
				\sum_{\overset{R\in \tilde{\mathcal{S}}}{R\subset Q}} \frac{\normadefpv{\car_R}{n/\delta}\eta(R)}{|R|}\cdot\car_{R}(x)\\
				&\lesssim \normadefp{b}{BMO^{\ex{\delta}}_\eta}
				\normadefpv{\ca}{n/\delta}
				\sum_{\overset{R\in \tilde{\mathcal{S}}}{R\subset Q}} \frac{\eta(R)}{|R|}\cdot\car_{R}(x).
			\end{align*}
		Therefore,
		\vspace{-0.7cm}
		
			\begin{align*}
				\int_{Q}|b(x)-b_Q|^k|f(x)|\,dx
				\lesssim
				\normadefp{b}{BMO^{\ex{\delta}}_\eta}^k 
				\normadefpv{\ca}{n/\delta}^{k}
				\int_{Q}
				\left(\sum_{\overset{R\in \tilde{\mathcal{S}}}{R\subset Q}} \frac{\eta(R)}{|R|}\cdot\car_{R}(x)\right)^k|f(x)|\,dx.
			\end{align*}
		Since the cubes from $\tilde{\mathcal{S}}$ are dyadic,
		\vspace{-0.2cm}
		
			\begin{align*}
				\left(\sum_{\overset{R\in \tilde{\mathcal{S}}}{R\subset Q}}
				\frac{\eta(R)}{|R|}\cdot\car_{R}(x)\right)^k
				&=\sum_{\overset{\{R_1,R_2,...,R_k\}\subset\tilde{\mathcal{S}}}{R_i\subset Q}}
				\left(\prod_{i=1}^{k}\frac{\eta(R_i)}{|R_i|}\right)\cdot\car_{R_1}(x)\car_{R_2}(x)...\car_{R_k}(x)\\
				&=\sum_{\overset{\{R_1,R_2,...,R_k\}\subset\tilde{\mathcal{S}}}{R_i\subset Q}}
				\left(\prod_{i=1}^{k}\frac{\eta(R_i)}{|R_i|}\right)\cdot\car_{R_1\cap R_2\cap...\cap R_k}(x)\\
				&\leq k!\sum_{\overset{\{R_1,R_2,...,R_k\}\subset\tilde{\mathcal{S}}}{R_k\subset R_{k-1}\subset...\subset R_1\subset Q}}
				\left(\prod_{i=1}^{k}\frac{\eta(R_i)}{|R_i|}\right)\cdot\car_{R_k}(x).
			\end{align*}
		Hence
		\vspace{-0.2cm}
		
			\begin{align*}
				&\int_{Q}
				\left(\sum_{\overset{R\in \tilde{\mathcal{S}}}{R\subset Q}} \frac{\eta(R)}{|R|}\cdot\car_{R}(x)\right)^k|f(x)|\,dx
				\lesssim
				\sum_{\overset{\{R_1,R_2,...,R_k\}\subset\tilde{\mathcal{S}}}{R_k\subset R_{k-1}\subset...\subset R_1\subset Q}}
				\left(\prod_{i=1}^{k}\frac{\eta(R_i)}{|R_i|}\right)\int_{R_k}|f(x)|\,dx\\
				&\quad\quad\quad\quad\quad\quad\quad=
				\sum_{\overset{\{R_1,R_2,...,R_k\}\subset\tilde{\mathcal{S}}}{R_k\subset R_{k-1}\subset...\subset R_1\subset Q}}
				\left(\prod_{i=1}^{k}\frac{\eta(R_i)}{|R_i|}\right)|f|_{R_k}|R_k|\\
				&\quad\quad\quad\quad\quad\quad\quad=
				\sum_{\overset{\{R_1,R_2,...,R_{k-1}\}\subset\tilde{\mathcal{S}}}{R_{k-1}\subset R_{k-2}\subset...\subset R_1\subset Q}}
				\left(\prod_{i=1}^{k-1}\frac{\eta(R_i)}{|R_i|}\right)
				\sum_{\overset{R_k\in\tilde{\mathcal{S}}}{R_k\subset R_{k-1}}}
				\int_{R_k}|f|_{R_k}\eta(x)\,dx\\
				&\quad\quad\quad\quad\quad\quad\quad= \sum_{\overset{\{R_1,R_2,...,R_{k-1}\}\subset\tilde{\mathcal{S}}}{R_{k-1}\subset R_{k-2}\subset...\subset R_1\subset Q}}
				\left(\prod_{i=1}^{k-1}\frac{\eta(R_i)}{|R_i|}\right)
				\int_{R_{k-1}}\left(\sum_{\overset{R_k\in\tilde{\mathcal{S}}}{R_k\subset R_{k-1}}}|f|_{R_k}\cdot\car_{R_k}(x)\right)\eta(x)\,dx\\
				&\quad\quad\quad\quad\quad\quad\quad\leq 
				\sum_{\overset{\{R_1,R_2,...,R_{k-1}\}\subset\tilde{\mathcal{S}}}{R_{k-1}\subset R_{k-2}\subset...\subset R_1\subset Q}}
				\left(\prod_{i=1}^{k-1}\frac{\eta(R_i)}{|R_i|}\right)
				\int_{R_{k-1}}\mathcal{A}_{\tilde{\mathcal{S}}}(f)(x)\eta(x)\,dx\\
				&\quad\quad\quad\quad\quad\quad\quad\leq 
				\sum_{\overset{\{R_1,R_2,...,R_{k-1}\}\subset\tilde{\mathcal{S}}}{R_{k-1}\subset R_{k-2}\subset...\subset R_1\subset Q}}
				\left(\prod_{i=1}^{k-1}\frac{\eta(R_i)}{|R_i|}\right)
				\int_{R_{k-1}}(\mathcal{A}_{\tilde{\mathcal{S}}})_\eta f(x)\,dx.
			\end{align*}
		Using this argument $k$ times we conclude
		\vspace{-0.6cm}
		
			\begin{align*}
				\int_{Q}
				\left(\sum_{\overset{R\in \tilde{\mathcal{S}}}{R\subset Q}} \frac{\eta(R)}{|R|}\cdot\car_{R}(x)\right)^k|f(x)|\,dx
				&\lesssim
				\int_{Q} (\mathcal{A}_{\tilde{\mathcal{S}}})^k_\eta f(x)\,dx.		
			\end{align*}
		We are done. 
	\end{proof}

	For $\ex{\beta}\in\mathcal{P}(\z)$ and $\mathcal{S}$ a sparse family, we define the variable fractional sparse operator $\mathcal{I}^{\ex{\beta}}_\mathcal{S}$ by  
	\vspace{-0.4cm}
	
		$$\mathcal{I}^{\ex{\beta}}_\mathcal{S}f(x)=\sum_{Q\in \mathcal{S}}\normadefpv{\ca}{\beta}f_Q\cdot\ca(x), \esp f\in \lplocc{1}.$$ 
	
	If $0<\alpha<n$ and $\ex{\beta}\equiv n/\alpha$, this operator was studied in \cite{CU} in the classical context of weighted Lebesgue spaces. 
	We are interested in studying the boundedness properties of the operators $\mathcal{I}^{\ex{\beta}}_\mathcal{S}$ on weighted variable Lebesgue spaces.
	
	The classes of weights we will be dealing with are a variable version of the well known $A_{p,q}$ classes of Muckenhoupt and Wheeden (see \cite{MuckenhouptWheeden}). 	Let $\ex{p},\ex{q}\in\mathcal{P}(\z)$.
	We say that a weight $w\in A_{\ex{p},\ex{q}}$ if there exists a positive constant $C$ such that, for every cube $Q\subset\z$,  
		\begin{equation}\label{Ap,q}
			\frac{\normadefpv{\ca w}{q}}{\normadefpv{\ca}{q}}\frac{\normadefpv{\ca w^{-1}}{p'}}{\normadefpv{\ca }{p'}}\leq C.
		\end{equation}
	The smallest of such constants will be denoted by $[w]_{A_{\ex{p},\ex{q}}}$.
	Note that $w\in A_{\ex{p},\ex{q}}$ is equivalent to $w^{-1}\in A_{\ex{q'},\ex{p'}}$. 
	When $\ex{p}=\ex{q}$, we obtain the $A_{\ex{p}}$ class given 
	in \cite{CDH} that characterizes the boundedness of the Hardy–Littlewood maximal operator on $\lpw{p}{w}(\z)$.

	We obtain the following boundedness result for $\mathcal{I}^{\ex{\beta}}_\mathcal{S}$ between variable weighted Lebesgue spaces.

	\begin{propo}\label{sparse fraccionario en lp variable con peso variable}
		Let $\ex{p},\ex{q}\inplogdern$ such that $1<p^-\le\ex{p}\leq\ex{q}\le q^+<\infty$. Let $\ex{\beta}$ be the exponent defined by $1/\ex{\beta}=1/\ex{p}-1/\ex{q}$ and $\mathcal{S}$ be a sparse family . 
		Assume $w\in A_{\ex{p},\ex{q}}$, then
		\vspace{-0.4cm}
		
		 $$\mathcal{I}^{\ex{\beta}}_{\mathcal{S}}:\lpwr{p}{w}\to\lpwr{q}{w}.$$
	\end{propo}

	Note that if $\ex{p}\equiv\ex{q}$, $w\in A_{\ex{p}}$ and $\mathcal{S}$ be a sparse family, from the proposition above we obtain that  		
		\begin{equation}\label{sparse en lp variable con peso variable}
			\mathcal{A}_{\mathcal{S}}:\lpwr{p}{w}\to\lpwr{p}{w}.
		\end{equation}
	This generalizes the well-known result proved in \cite{CMPc} or \cite{LN} for the sparse operator $\mathcal{A}_\mathcal{S}$ in the classical context. 
\medskip

	In order to prove the Proposition \ref{sparse fraccionario en lp variable con peso variable} let see some useful properties of the classes $A_{\ex{p},\ex{q}}$.
\medskip

	Note that if $\ex{p},\ex{q}\inplogdern$, $\ex{p}\le \ex{q}$, the opposite inequality of \eqref{Ap,q} follows by H\"{o}lder's inequality \eqref{holderrpq} and Lemma \ref{equivalenciabeta}, so we have that 
	\vspace{-0.2cm}
	
		$$\frac{\normadefpv{\ca w}{q}}{\normadefpv{\ca}{q}}\frac{\normadefpv{\ca w^{-1}}{p'}}{\normadefpv{\ca }{p'}}\simeq1$$
	if $w\in A_{\ex{p},\ex{q}}$.	
\medskip

	Let $\ex{p},\ex{q}\inplogdern$ such taht  $\ex{p}\le \ex{q}$, then \vspace{-0.4cm}
	
		\begin{align}\label{desigualdad promedios}
			\frac{\normadefpv{\ca f}{p}}
			{\normadefpv{\ca}{p}}
			\lesssim 
			\frac{\normadefpv{\ca f}{q}}
			{\normadefpv{\ca}{q}},\esp f\in \lplocc{1}.
		\end{align}
	Indeed, let $\ex{\beta}$ be defined by $1/\ex{\beta}=1/\ex{p}-1/\ex{q}$. 
	Then $\ex{\beta}\inplogdern$ and, by H\"{o}lder's inequality \eqref{holderrpq} and Lemma \ref{equivalenciabeta} we obtain \eqref{desigualdad promedios}. 
	
\begin{lema}\label{Prop pesos Apq}
	Let $\ex{p},\ex{q}\inplogdern$ such that $\ex{p}\le \ex{q}$ and $w\in A_{\ex{p},\ex{q}}$.
	Then \vspace{-0.3cm}
	
		\begin{description}
			\item[(i)]\label{Ap,q implica Ap y Aq} 
				$w\in A_{\ex{p}}\cap A_{\ex{q}}$		
			\item[(ii)]\label{Ap,q implica Aq,p} 
				$w\in A_{\ex{q},\ex{p}}$. Moreover, 
				\vspace{-0.2cm}
				
					$$	\frac{\normadefpv{\ca w}{p}}				{\normadefpv{\ca}{p}}
					\frac{\normadefpv{\ca w^{-1}}{q'}}				{\normadefpv{\ca }{q'}}\simeq1.$$
		\end{description}
\end{lema}
	\begin{proof}
		Let us see \textbf{(i)}. Since $\ex{p}\le \ex{q}$, by \eqref{desigualdad promedios} we have
		\vspace{-0.4cm}
				
			\begin{align*}
				\frac{\normadefpv{\ca w}{p}}
				{\normadefpv{\ca}{p}}
				\frac{\normadefpv{\ca w^{-1}}{p'}}
				{\normadefpv{\ca }{p'}}
				&\lesssim 
				\frac{\normadefpv{\ca w}{q}}
				{\normadefpv{\ca}{q}}
				\frac{\normadefpv{\ca w^{-1}}{p'}}
				{\normadefpv{\ca }{p'}}
				\leq [w]_{A_{\ex{p},\ex{q}}}.
			\end{align*}
		In the same way, since $\ex{q'}\le \ex{p'}$, we have
		\vspace{-0.4cm}
		
			\begin{align*}
				\frac{\normadefpv{\ca w}{q}}
				{\normadefpv{\ca}{q}}
				\frac{\normadefpv{\ca w^{-1}}{q'}}
				{\normadefpv{\ca }{q'}}
				&\lesssim 
				\frac{\normadefpv{\ca w}{q}}
				{\normadefpv{\ca}{q}}
				\frac{\normadefpv{\ca w^{-1}}{p'}}
				{\normadefpv{\ca }{p'}}
				\leq [w]_{A_{\ex{p},\ex{q}}}.
			\end{align*}
			
		In order to prove \textbf{(ii)}, by applying H\"{o}lder's inequality \eqref{holderpp} and Lemma \ref{equivalenciabeta} we have
		\vspace{-0.4cm}
		
			\begin{align*}
				1&\lesssim 
				\frac{\normadefpv{\ca w}{p}}
				{\normadefpv{\ca}{p}}
				\frac{\normadefpv{\ca w^{-1}}{p'}}
				{\normadefpv{\ca }{p'}}
				\frac{\normadefpv{\ca w}{q}}
				{\normadefpv{\ca}{q}}
				\frac{\normadefpv{\ca w^{-1}}{q'}}
				{\normadefpv{\ca }{q'}}.
			\end{align*}
			Then, by the assumptions on the weight, \eqref{desigualdad promedios} and \textbf{(i)} we obtain
			\vspace{-0.4cm}
			
			\begin{align*}
				1&\lesssim [w]_{A_{\ex{p},\ex{q}}}
				\frac{\normadefpv{\ca w}{p}}
				{\normadefpv{\ca}{p}}
				\frac{\normadefpv{\ca w^{-1}}{q'}}
				{\normadefpv{\ca }{q'}}\\
				&\lesssim [w]_{A_{\ex{p},\ex{q}}}
				\frac{\normadefpv{\ca w}{q}}
				{\normadefpv{\ca}{q}}
				\frac{\normadefpv{\ca w^{-1}}{q'}}
				{\normadefpv{\ca }{q'}}
				\lesssim [w]_{A_{\ex{p},\ex{q}}}[w]_{A_{\ex{q}}}.\end{align*}\end{proof}
	The following proposition provides us with an ``openness" type property of class $A_{\ex{p},\ex{q}}$.
\begin{propo}\label{apertura Ap,q}
	Let $\ex{p},\ex{q}\inplogdern$ such that $1<p^-\le \ex{p}\le\ex{q}\le q^+<\infty$ and $w\in A_{\ex{p},\ex{q}}$. Then there exist $\ex{u},\ex{v}\inplogdern$ such that $(p/u)^->1$, 
	$(q'/v')^->1$ and $w\in A_{\ex{u},\ex{v}}$.
\end{propo} 

	For the case $\ex{p}\equiv\ex{q}$, this proposition was proved in \cite{CDH}.
\begin{proof}[Proof of Proposition \ref{apertura Ap,q}]
	Since $w\in A_{\ex{p},\ex{q}}$, by Lemma \ref*{Prop pesos Apq}\textbf{(i)}, we have that $w\in A_{\ex{p}}$. 
	Similarly, since $w^{-1}\in A_{\ex{q'},\ex{p'}}$, $w^{-1}\in A_{\ex{q'}}$. 
	Then by Lemma \ref{s de la apertura clase Ap}, since $p^->1$ and $(q')^->1$, there exist two constants $s\in(1/p^-,1)$ and $r\in(1/(q')^-,1)$ such that 
		\begin{equation}\label{equ}
		w^{1/s}\in A_{s\ex{p}}\esp\text{and}\esp w^{-1/r}\in A_{r\ex{q'}}.
		\end{equation}
	We denote $\ex{u'}=\frac{1}{s}{(s\ex{p})'}$ and $\ex{v}=\frac{1}{r}{(r\ex{q'})'}$.
	Note that
	\vspace{-0.2cm}
	 
		$$\frac{\ex{p}}{\ex{u}}=\ex{p}(1-s)+1\geq (p)^-(1-s)+1$$
		\vspace{-0.2cm}
	and 
	\vspace{-0.2cm}
	
		$$\frac{\ex{q'}}{\ex{v'}}=\ex{q'}(1-r)+1\geq (q')^-(1-r)+1,$$
	so that $(p/u)^->1$ and $(q'/v')^->1$.
	By \eqref{equ} and Lemma \ref{3.2.6}, we have that 
	\vspace{-0.2cm}
	
		$$\frac{\normadefpv{\ca w}{p}}
		{\normadefpv{\ca}{p}}
		\frac{\normadefpv{\ca w^{-1}}{u'}}
		{\normadefpv{\ca }{u'}}
		\simeq 1
		\esp\text{and}\esp
		\frac{\normadefpv{\ca w}{v}}
		{\normadefpv{\ca}{v}}
		\frac{\normadefpv{\ca w^{-1}}{q'}}
		{\normadefpv{\ca }{q'}}
		\simeq 1
		$$
	respectively.
	Thus, by Lemma \ref*{Prop pesos Apq}\textbf{(ii)} we obtain $w\in A_{\ex{u},\ex{v}}$.
\end{proof}

	We can now proceed with the proof of Proposition \ref{sparse fraccionario en lp variable con peso variable}. 
		
\begin{proof}[Proof of Proposition~\ref{sparse fraccionario en lp variable con peso variable}]		
	By duality and since $\mathcal{S}$ is sparse we have
	\vspace{-0.4cm}
	
	\begin{align*}
	\normadeflpw{\mathcal{I}^{\ex{\beta}}_{\mathcal{S}}f}{q}{w}
	&=\sup_{\normadeflpw{g}{q'}{{w^{-1}}}\leq1}
	\int_{\z}g(x)\mathcal{I}^{\ex{\beta}}_{\mathcal{S}}\hspace{-0.1cm}f(x)\,dx
	=\sup_{\normadeflpw{g}{q'}{{w^{-1}}}\leq1}
	\sum_{Q\in \mathcal{S}} |Q|\normadefpv{\ca}{\beta}f_Q\,g_Q\\
	&\lesssim\sup_{\normadeflpw{g}{q'}{{w^{-1}}}\leq1}
	\sum_{Q\in \mathcal{S}} |E(Q)|\normadefpv{\ca}{\beta}f_Q\,g_Q.
	\end{align*}
	Let $\ex{u},\ex{v}$ the exponents provided by Proposition \ref{apertura Ap,q}.
	Then by H\"{o}lder's inequality and Lemma \ref{p en plog} we obtain 	
	\vspace{-0.7cm}
	
	\begin{align*}
	\normadeflpw{\mathcal{I}^{\ex{\beta}}_{\mathcal{S}}f}{q}{w}
	&\lesssim\sup_{\normadeflpw{g}{q'}{{w^{-1}}}\leq1}
	\sum_{Q\in\mathcal{S}}|E(Q)|
	{\normadefpv{\ca}{\beta}}
	\frac{\normadefpv{\ca f w}{u}}
	{\normadefpv{\ca}{u}}
	\frac{\normadefpv{\ca w^{-1}}{u'}}
	{\normadefpv{\ca }{u'}}
	\frac{\normadefpv{\ca gw^{-1}}{v'}}
	{\normadefpv{\ca}{v'}}
	\frac{\normadefpv{\ca w}{v}}
	{\normadefpv{\ca }{v}}
	\\
	&\lesssim[w]_{A_{\ex{u},\ex{v}}}\sup_{\normadeflpw{g}{q'}{{w^{-1}}}\leq1}
	\sum_{Q\in\mathcal{S}}|E(Q)|
	{\normadefpv{\ca}{\beta}}
	\frac{\normadefpv{\ca f w}{u}}
	{\normadefpv{\ca}{u}}
	\frac{\normadefpv{\ca gw^{-1}}{v'}}
	{\normadefpv{\ca}{v'}}
	\\
	&\lesssim\sup_{\normadeflpw{g}{q'}{{w^{-1}}}\leq1}
	\int_{\z} M_{\ex{\beta},L^{\ex{u}}}(fw)(x)M_{L^{\ex{v'}}}(gw^{-1})(x)\,dx\\
	&\lesssim\sup_{\normadeflpw{g}{q'}{{w^{-1}}}\leq1}
	\normadefpv{M_{\ex{\beta},L^{\ex{u}}}(fw)}{q}\normadefpv{M_{L^{\ex{v'}}}(gw^{-1})}{q'}\\
	&\lesssim\sup_{\normadeflpw{g}{q'}{{w^{-1}}}\leq1}
	\normadefpv{fw}{p}\normadefpv{gw^{-1}}{q'}
	\le \normadeflpw{f}{p}{w}
	\end{align*}
	where we have used that, by Theorem \ref{teo logl}, $M_{L^{\ex{v'}}}:\lpr{q'}\hookrightarrow\lpr{q'}$ and by Theorem \ref{teo max llogl}, $\mbetas{\beta}{L^{\ex{u}}}:\lpr{p}\hookrightarrow\lpr{q}$.
\end{proof}

	The following lemma will be useful in the proof of the Theorem \ref{teo simb lips bloom}. 
	\begin{lema}\label{factores en Apq implica factores en Apq}
		Let $\ex{p},\ex{q}\in\mathcal{P}(\z)$, $\mu,\lambda\in A_{\ex{p},\ex{q}}$ and $\nu={\mu}/{\lambda}$. 
		Then $\lambda\nu^{\frac{m-h}{m}}\in A_{\ex{p},\ex{q}}$ for every $m\in\zN$ and for each $h=0,1,...,m$.
	\end{lema} \vspace{-0.5cm}

	Let $m\in\zN$ y $k\in\{1,2,...,m\}$.
	Note that, if $\ex{p},\ex{q}\in\mathcal{P}(\z)$ be such that $\ex{p}\le\ex{q}$, $\mu,\lambda\in A_{\ex{p},\ex{q}}$, $\nu={\mu}/{\lambda}$
	and we denote $\eta=\nu^{1/m}$ we have
		\begin{equation}\label{op sparse con norma q y pesos }
		\normadeflpw{(\mathcal{A}_\mathcal{S})^{k}_{\eta}F}{q}{{\lambda}}
		\lesssim \normadeflpw{f}{q}{{\lambda \eta^{k}}}, \esp \forall\,F\in\lplocc{1}
		\end{equation}
	and
		\begin{equation}\label{op sparse con norma p y pesos }
		\normadeflpw{(\mathcal{A}_\mathcal{S})^{k}_{\eta}F}{p}{{\lambda\eta^{m-k}}}
		\lesssim \normadeflpw{f}{p}{{\lambda \eta^{m}}}, \esp \forall\,F\in\lplocc{1}.
		\end{equation}
%
	\begin{proof}[Proof of Lemma \ref{factores en Apq implica factores en Apq}]
		By H\"{o}lder's inequality \eqref{holderrpq} and Lemma \ref{3.2.6}, we have
		\vspace{-0.2cm}
		
		\begin{align*}
		&\frac{\normadefpv{\ca \lambda\nu^{\frac{m-h}{m}}}{q}}{\normadefpv{\ca}{q}}
		\frac{\normadefpv{\ca \lambda^{-1}\nu^{-\frac{m-h}{m}}}{p'}}{\normadefpv{\ca}{p'}}	=\frac{\normadefpv{\ca \lambda^{\frac{h}{m}}\mu^{\frac{m-h}{m}}}{q}}{\normadefpv{\ca}{q}}
		\frac{\normadefpv{\ca \lambda^{-\frac{h}{m}}\mu^{-\frac{m-h}{m}}}{p'}}{\normadefpv{\ca}{p'}}\\
		&\quad\lesssim
		\frac{\normadefp{\ca\lambda^{\frac{h}{m}}}{\frac{m\ex{q}}{h}}}
			{\normadefp{\ca}{\ex{q}}^{\frac{h}{m}}}
		\frac{\normadefp{\ca \mu^{\frac{m-h}{m}}}{\frac{m\ex{q}}{m-h}}}
			{\normadefp{\ca}{\ex{q}}^{\frac{m-h}{m}}}
		\frac{\normadefp{\ca \lambda^{-\frac{h}{m}}}{\frac{m\ex{p'}}{h}}}
			{\normadefp{\ca}{\ex{p'}}^{\frac{h}{m}}}
		\frac{\normadefp{\ca \mu^{-\frac{m-h}{m}}}{\frac{m\ex{p'}}{m-h}}}
			{\normadefp{\ca}{\ex{p'}}^{\frac{m-h}{m}}}\\
		&\quad=\left(\frac{\normadefpv{\ca \lambda}{q}}{\normadefpv{\ca}{q}}\right)^{\frac{h}{m}}
		\left(\frac{\normadefpv{\ca \mu}{q}}{\normadefpv{\ca}{q}}\right)^{\frac{m-h}{m}}
		\left(\frac{\normadefpv{\ca \lambda^{-1}}{p'}}{\normadefpv{\ca}{p'}}\right)^{\frac{h}{m}}
		\left(\frac{\normadefpv{\ca \mu^{-1}}{p'}}{\normadefpv{\ca}{p'}}\right)^{\frac{m-h}{m}}\\
		&\quad\leq [\lambda]_{A_{\ex{p},\ex{q}}}^{\frac{h}{m}}[\mu]_{A_{\ex{p},\ex{q}}}^{\frac{m-h}{m}}.
		\end{align*}
	\end{proof}		
	
\section{Proofs of main results}
\begin{proof}[Proof of Theorem~\ref{conmutador simbolo lip general}]
	Since $v\in \lploc{p}$ implies that the set of bounded functions with compact support is dense in $\lpw{p}{v}(\z)$ and taking into account Theorem \ref{est puntual del conmutador de orden m}, it is enough to show that for a sparse family $\mathcal{S}$, 
	\vspace{-0.2cm}
	
	$$\normadeflpw{\mathcal{A}_{{\mathcal{S}}}^{m,h}(b,f)}{p}{w}
	\lesssim \normadeflpw{f}{p}{v}\esp h\in\{0,1,...,m\}$$
	for each nonnegative bounded function with compact support $f$. 
	Let $h\in\{0,1,...,m\}$, by duality
	
	\begin{align}\label{R2}
		\normadeflpw{\mathcal{A}_{{\mathcal{S}}}^{m,h}(b,f)}{p}{w}
		&\lesssim\sup_{\normadeflpw{g}{p'}{{w^{-1}}}\leq1}
		\int_{\z}g(x)\mathcal{A}_{{\mathcal{S}}}^{m,h}(b,f)(x)\,dx 
		\nonumber\\
		&=\sup_{\normadeflpw{g}{p'}{{w^{-1}}}\leq1}
		\sum_{Q\in \mathcal{S}}
		|Q|
		\frac{1}{|Q|}\int_Q|b(x)-b_{Q}|^{h}f(x)\,dx\,
		\frac{1}{|Q|}\int_Q|b(x)-b_{Q}|^{m-h}g(x)\,dx. 
	\end{align}	
	Let us denote $\ex{s}= R\ex{p'}$ and $\ex{l}=S\ex{p}$.
	Since $(p')^+<R(p')^-$ and $p^+<Sp^-$, $(s')^+<p^-$ and $(l')^+<(p^+)'$ then, we can take two constants $\mathrm{A},\mathrm{B}$ such that
	\vspace{-0.4cm}
	
	$$(s')^+< \mathrm{A} <p^-
	\esp\esp\texto{ and }\esp\esp
	(l')^+<\mathrm{B}<(p^+)',$$
	and $\omega(\cdot),\tau(\cdot)$ defined by
	\vspace{-0.3cm}
	
	$$\frac{1}{\omega(\cdot)}=\frac{1}{s(\cdot)}+\frac{1}{\mathrm{A}}
	\esp\esp\texto{ and }\esp\esp
	\frac{1}{\tau(\cdot)}=\frac{1}{l(\cdot)}+\frac{1}{\mathrm{B}}.$$
	Observe that $\ex{\omega},\ex{\tau}\inplogdern$ since $\ex{s},\ex{l}\inplogdern$.
	
	On the other hand note that, for $k$ an integer, an exponent $\ex{r}\inplogdern$ with $1<r^-\le r^+<\infty$ and $H\in\lplocc{1}$, by H\"{o}lder's inequality \eqref{holderpp}, Lemmas \ref{p en plog} and \ref{Izuki con lipschitz general}, we have 
	\vspace{-0.3cm}
	
		\begin{align*}
			\frac{1}{|Q|}\int_Q |b(y)-b_Q|^kH(y)\,dy 
			&\lesssim \frac{\normadefpv{\ca|b-b_Q|^k}{r'}}{\normadefpv{\ca}{r'}}
			\frac{\normadefpv{\ca H}{r}}{\normadefpv{\ca}{r}}
			\lesssim a(Q)^k\normadefp{b}{\zLL_a}^k
			\frac{\normadefpv{\ca H}{r}}{\normadefpv{\ca}{r}}.
		\end{align*}
	
	Thus, by \eqref{R2}, we have
	\vspace{-0.4cm}
	
	\begin{align*}
	\normadeflpw{\mathcal{A}_{{\mathcal{S}}}^{m,h}(b,f)}{p}{w}
	&\lesssim\normadefp{b}{\zLL_a}^m
	\sup_{\normadeflpw{g}{p'}{{w^{-1}}}\leq1}
	\sum_{Q\in \mathcal{S}}|Q|
	a(Q)^m
	\frac{\normadefpv{\car_{Q}f}{\omega}}
	{\normadefpv{\car_{Q}}{\omega}}
	\frac{\normadefpv{\car_{Q}g}{\tau}}
	{\normadefpv{\car_{Q}}{\tau}}.
	\end{align*}
	Using that $\mathcal{S}$ is a sparse family and H\"{o}lder's inequality \eqref{holderrpq},  Lemma \ref{p en plog} and the hypothesis on the weights we obtain that
	\vspace{-0.6cm}
	
	\begin{align*}
	&\normadeflpw{\mathcal{A}_{{\mathcal{S}}}^{m,h}(b,f)}{p}{w}\\
	&\quad\quad\lesssim\normadefp{b}{\zLL_a}^m
	\sup_{\normadeflpw{g}{p'}{{w^{-1}}}\leq1}
	\sum_{Q\in S}|E(Q)|
	a(Q)^m
	\frac{\normadeflpl{\car_{Q}fv}{\mathrm{A}}}{\normadeflpl{\car_{Q}}{\mathrm{A}}}
	\frac{\normadefpv{\car_{Q}v^{-1}}{s}}{\normadefpv{\car_{Q}}{s}}
	\frac{\normadeflpl{\car_{Q}gw^{-1}}{\mathrm{B}}}
	{\normadeflpl{\car_{Q}}{\mathrm{B}}}
	\frac{\normadefpv{\car_{Q}w}{l}}{\normadefpv{\car_{Q}}{l}}\\
	&\quad\quad\lesssim \normadefp{b}{\zLL_a}^m
	\sup_{\normadeflpw{g}{p'}{{w^{-1}}}\leq1}
	\int_{\z}M_{L^\mathrm{A}}(fv)(x)M_{L^\mathrm{B}}(gw^{-1})(x)\,dx\\
	&\quad\quad\lesssim 
	\normadefp{b}{\zLL^1_a}^m
	\sup_{\normadeflpw{g}{p'}{{w^{-1}}}\leq1}
	\normadefpv{M_{L^\mathrm{A}}(fv)}{p}
	\normadefpv{M_{L^\mathrm{B}}(gw^{-1})}{p'}\\
	&\quad\quad\lesssim \normadefp{b}{\zLL_a}^m
	\sup_{\normadeflpw{g}{p'}{{w^{-1}}}\leq1}
	\normadefpv{fv}{p}\normadefpv{gw^{-1}}{p'}
	\leq \normadefp{b}{\zLL_a}^m \normadeflpw{f}{p}{v}		
	\end{align*}
	where we have used that by Theorem \ref{teo logl},  $M_{L^\mathrm{A}}:\lpr{p}\hookrightarrow\lpr{p}$ since $\mathrm{A}<p^-$ and $M_{L^\mathrm{B}}:\lpr{p'}\hookrightarrow\lpr{p'}$ since $\mathrm{B}<(p')^-$.	
\end{proof}	

\begin{proof}[Proof of Theorem~\ref{conmutador simbolo lip variable}]
	As before it suffices to provide suitable estimates for 
	\vspace{-0.6cm}
	
	\begin{align*}
	\sum_{Q\in \mathcal{S}}
	|Q|
	(|b-b_{Q}|^{h}f)_{Q}
	(|b-b_{Q}|^{m-h}g)_{Q}, \esp h\in\{0,1,...,m\} 
	\end{align*}
	for each nonnegative bounded function with compact support $f$ and each $g$ with $\normadeflpw{g}{p'}{{w^{-1}}}\leq1$,	where $\mathcal{S}$ is a sparse family. Note that, for $k$ an non negative integer and $H\in\lplocc{1}$, by Lemma \ref{obs teorema delta variable} we have 
	\vspace{-0.9cm}
	
		\begin{align*}		
			\frac{1}{|Q|}\int_Q 	|b(x)-b_Q|^kH(x)\,dx 
			&\lesssim  \normadefp{\ca}{n/\ex{\delta}}^k H_Q. 
		\end{align*}	
	Then we have
	\begin{align}\label{9}
	\sum_{Q\in \mathcal{S}}
	|Q|
	(|b-b_{Q}|^{h}f)_{Q}
	(|b-b_{Q}|^{m-h}g)_{Q}
	\lesssim 
	\sum_{Q\in \mathcal{S}}|Q|
	\normadefp{\car_{Q}}{n/\ex{\delta}}^m
	f_{Q}\,g_{Q}.
	\end{align}
	By condition $\mathcal{F}$ and H\"{o}lder's inequalities \eqref{holdermusie} and \eqref{holdermusie2} we have
	\vspace{-0.5cm}
	
	\begin{align*}
	f_{Q}
	&\lesssim
	\frac{\normadeflpl{\car_{Q} f}{\mathsf{D}}}
	{\normadeflpl{\car_{ Q}}{\mathsf{D}}}
	\frac{\normadeflpl{\car_{Q}}{\mathsf{D}^*}}
	{\normadeflpl{\car_{ Q}}{\mathsf{D}^*}}
	\lesssim
	\frac{\normadeflpl{\car_{ Q} fv}{\mathsf{B}}}
	{\normadeflpl{\car_{ Q}}{\mathsf{B}}}
	\frac{\normadeflpl{\car_{Q} v^{-1}}{\mathsf{A}}}
	{\normadeflpl{\car_{ Q}}{\mathsf{A}}}
	\end{align*}
	and
	\vspace{-0.5cm}
	
	\begin{align*}
	g_{Q}
	&\lesssim
	\frac{\normadeflpl{\car_{Q} g}{\mathsf{J}}}
	{\normadeflpl{\car_{ Q}}{\mathsf{J}}}
	\frac{\normadeflpl{\car_{Q}}{\mathsf{J}^*}}
	{\normadeflpl{\car_{ Q}}{\mathsf{J}^*}}
	\lesssim
	\frac{\normadeflpl{\car_{Q} gw^{-1}}{\mathsf{H}}}
	{\normadeflpl{\car_{ Q}}{\mathsf{H}}}
	\frac{\normadeflpl{\car_{Q} w}{\mathsf{E}}}
	{\normadeflpl{\car_{Q}}{\mathsf{E}}}.
	\end{align*}
	Then from (\ref{9}), since $\mathcal{S}$ is sparse, by  the hypothesis on the weights and \eqref{hipotesis maximal} we have 
	\vspace{-0.6cm}
			
	\begin{align*}
	&\sum_{Q\in \mathcal{S}}
	|Q|
	(|b-b_{Q}|^{h}f)_{Q}
	(|b-b_{Q}|^{m-h}g)_{Q}\\
	&\quad\quad\lesssim
	\sum_{Q\in \mathcal{S}}|E(Q)|
	\normadefp{\car_{Q}}{n/\ex{\delta}}^m
	\frac{\normadeflpl{\car_{ Q} fv}{\mathsf{B}}}
	{\normadeflpl{\car_{ Q}}{\mathsf{B}}}			\frac{\normadeflpl{\car_{Q} v^{-1}}{\mathsf{A}}}
	{\normadeflpl{\car_{ Q}}{\mathsf{A}}}				\frac{\normadeflpl{\car_{Q} gw^{-1}}{\mathsf{H}}}
	{\normadeflpl{\car_{ Q}}{\mathsf{H}}}				\frac{\normadeflpl{\car_{Q} w}{\mathsf{E}}}				{\normadeflpl{\car_{Q}}{\mathsf{E}}}\\				&\quad\quad\lesssim				\int_{\z}M_{\mathsf{B}(\cdot,L)}(fv)(x)
	M_{\mathsf{H}(\cdot,L)}(gw^{-1})(x)\,dx
	\lesssim 			\normadefpv{M_{\mathsf{B}(\cdot,L)}(fv)}{p}\normadefpv{M_{\mathsf{H}(\cdot,L)}(gw^{-1})}{p'}\\
	&\quad\quad\lesssim
	\normadefpv{fv}{p}\normadefpv{gw^{-1}}{p'}
	\lesssim \normadeflpw{f}{p}{v}.
	\end{align*}
\end{proof}

\begin{proof}[Proof of Theorem~\ref{teo simb lips bloom}]
	Taking into account Theorem \ref{est puntual del conmutador de orden m}, it suffices to prove the estimate for the sparse operators
	\vspace{-0.6cm}
	
		\begin{align*}
		\mathcal{A}^{m,h}_{\mathcal{S}}(b,f)(x)
		&=\sum_{Q\in \mathcal{S}}
		|b(x)-b_{Q}|^{m-h}
		(|b-b_{Q}|^{h}|f|)_{Q}\cdot\ca(x), \esp h\in\{0,1,...,m\},
		\end{align*}
		for each nonnegative bounded function with compact support $f$. By duality we have \vspace{-0.6cm}
	
		\begin{align*}\normadeflpw{\mathcal{A}_{{\mathcal{S}}}^{m,h}(b,f)}{q}{\lambda}
		&\lesssim\sup_{\normadefpv{g}{q'}\leq1}
		\int_{\z}\lambda(x)g(x)\mathcal{A}_{{\mathcal{S}}}^{m,h}(b,f)(x)\,dx \nonumber\\
		&=\sup_{\normadefpv{g}{q'}\leq1}
		\sum_{Q\in \mathcal{S}}(|b-b_{Q}|^{h}|f|)_{Q}\int_{Q}|b(x)-b_{Q}|^{m-h}\lambda(x)g(x)\,dx. 
		\end{align*}
	Let $\tilde{\mathcal{S}}$ be the sparse family provided by Proposition \ref{integral de b por f menor q la integral del operador sparse} and $Q\in\tilde{\mathcal{S}}$. 
	Then, by this proposition and denoting $\eta=\nu^{1/m}$, we have
	\vspace{-0.6cm}
	
		\begin{align*}
			\normadeflpw{\mathcal{A}_{{\mathcal{S}}}^{m,h}(b,f)}{q}{\lambda}
			&\lesssim\sup_{\normadefpv{g}{q'}\leq1}
			\normadefp{b}{BMO^{\ex{\delta}}_\eta}^m
			\sum_{Q\in \tilde{\mathcal{S}}}\frac{\normadefpv{\ca}{n/\delta}^m}{|Q|}\int_{Q} (\mathcal{A}_{\tilde{\mathcal{S}}})^h_\eta f(x)\,dx
			\int_{Q} (\mathcal{A}_{\tilde{\mathcal{S}}})^{m-h}_\eta (\lambda g)(x)\,dx.
		\end{align*}
	Noting that $\normadefpv{\ca}{n/\delta}^m=\normadefpv{\ca}{n/m\delta}$, denoting $\ex{\beta}=n/m\ex{\delta}$ and $\mathcal{A}_{\tilde{\mathcal{S}}}=\mathcal{A}$, we have
	\vspace{-0.6cm}
	
	\begin{align}\label{misma estimacion}
		\normadeflpw{\mathcal{A}_{{\mathcal{S}}}^{m,h}(b,f)}{q}{\lambda}
		&\lesssim\sup_{\normadefpv{g}{q'}\leq1}
			\normadefp{b}{BMO^{\ex{\delta}}_\eta}^m
			\sum_{Q\in \tilde{\mathcal{S}}}
			\frac{\normadefpv{\ca}{\beta}}{|Q|}
			\int_{Q} \mathcal{A}^h_\eta f(x)\,dx
			\int_{Q} \mathcal{A}^{m-h}_\eta (\lambda g)(x)\,dx\\
		&=\sup_{\normadefpv{g}{q'}\leq1}
			\normadefp{b}{BMO^{\ex{\delta}}_\eta}^m
			\int_{\z}\left(\sum_{Q\in \tilde{\mathcal{S}}}
			\normadefpv{\ca}{\beta}
			(\mathcal{A}^h_\eta f)_Q\cdot\ca(x)\right) \mathcal{A}^{m-h}_\eta (\lambda g)(x)\,dx\nonumber\\
		&=\sup_{\normadefpv{g}{q'}\leq1}
			\normadefp{b}{BMO^{\ex{\delta}}_\eta}^m
			\int_{\z}\mathcal{I}_{\tilde{\mathcal{S}}}^{\ex{\beta}}
			\left(\mathcal{A}^h_\eta f\right)\hspace{-1mm}(x)\,\,
			\mathcal{A}^{m-h}_\eta (\lambda g)(x)\,dx.\nonumber
	\end{align}
	Using $m-h$ times that $\mathcal{A}$ is self-adjoint, we have 
	\vspace{-0.4cm}
	
	\begin{align*}
		\int_{\z}\mathcal{I}^{\ex{\beta}}_{\tilde{\mathcal{S}}}
			\left(\mathcal{A}^h_\eta f\right)\hspace{-1mm}(x) \,\,
			\mathcal{A}^{m-h}_\eta (\lambda g)(x)\,dx
			=\int_{\z}
			\mathcal{A}
			\left(\mathcal{A}^{m-h-1}_\eta
			\left[\left(\mathcal{I}^{\ex{\beta}}_{\tilde{\mathcal{S}}}\right)_\eta\hspace{-1mm}\left(\mathcal{A}^h_\eta f\right)\right]\right)\hspace{-1mm}(x)  
			\,\lambda(x) g(x)\,dx.
	\end{align*}
	Combining the preceding estimates and H\"{o}lder's inequality, we obtain that 
	\vspace{-0.5cm}
		
	\begin{align*}
		\normadeflpw{\mathcal{A}_{{\mathcal{S}}}^{m,h}(b,f)}{q}{\lambda}
		&\lesssim \normadefp{b}{BMO^{\ex{\delta}}_\eta}^m \sup_{\normadeflp{g}{q'}\leq1}
		\int_{\z}\mathcal{A}
		\left(\mathcal{A}^{m-h-1}_\eta
		\left[\left(\mathcal{I}^{\ex{\beta}}_{\tilde{\mathcal{S}}}\right)_\eta\hspace{-1mm}\left(\mathcal{A}^h_\eta f\right)\right]\right)\hspace{-1mm}(x)  
		\,\lambda(x) g(x)\,dx\\ 
		&\lesssim \normadefp{b}{BMO^{\ex{\delta}}_\eta}^m
		\sup_{\normadeflp{g}{q'}\leq1}
		\normadefpv{\mathcal{A}
			\left(\mathcal{A}^{m-h-1}_\eta
			\left[\left(\mathcal{I}^{\ex{\beta}}_{\tilde{\mathcal{S}}}\right)_\eta\hspace{-1mm}\left(\mathcal{A}^h_\eta f\right)\right]\right)
			\lambda}{q}
		\normadefpv{g}{q'}\\ 
		&\lesssim \normadefp{b}{BMO^{\ex{\delta}}_\eta}^m
		\normadeflpw{\mathcal{A}
			\left(\mathcal{A}^{m-h-1}_\eta
			\left[\left(\mathcal{I}^{\ex{\beta}}_{\tilde{\mathcal{S}}}\right)_\eta\hspace{-1mm}\left(\mathcal{A}^h_\eta f\right)\right]\right)
			}{q}{\lambda}\\
		&\lesssim \normadefp{b}{BMO^{\ex{\delta}}_\eta}^m
		\normadeflpw{\mathcal{A}^{m-h-1}_\eta
			\left[\left(\mathcal{I}^{\ex{\beta}}_{\tilde{\mathcal{S}}}\right)_\eta\hspace{-1mm}\left(\mathcal{A}^h_\eta f\right)\right]
		}{q}{\lambda}
	\end{align*}
	where we have used that $\lambda\in A_{\ex{q}}$ and \eqref{sparse en lp variable con peso variable}. By inequality \eqref{op sparse con norma q y pesos } we obtain
	\vspace{-0.5cm}
	
	\begin{align*}
		\normadeflpw{\mathcal{A}_{{\mathcal{S}}}^{m,h}(b,f)}{q}{\lambda}
		&\lesssim \normadefp{b}{BMO^{\ex{\delta}}_\eta}^m
		\normadeflpw{\left(\mathcal{I}^{\ex{\beta}}_{\tilde{\mathcal{S}}}\right)_\eta\hspace{-1mm}\left(\mathcal{A}^h_\eta f\right)}{q}{{\lambda\eta^{m-h-1}}}
		= \normadefp{b}{BMO^{\ex{\delta}}_\eta}^m
		\normadeflpw{\mathcal{I}^{\ex{\beta}}_{\tilde{\mathcal{S}}}\hspace{-1mm}\left(\mathcal{A}^h_\eta f\right)}
		{q}{{\lambda\eta^{m-h}}}.
	\end{align*}
	Using Lemma \ref{factores en Apq implica factores en Apq} and Proposition \ref{sparse fraccionario en lp variable con peso variable} we obtain that
	\vspace{-0.5cm}
	
	\begin{align*}
		\normadeflpw{\mathcal{A}_{{\mathcal{S}}}^{m,h}(b,f)}{q}{\lambda}
		&\lesssim \normadefp{b}{BMO^{\ex{\delta}}_\eta}^m
		\normadeflpw{\mathcal{A}^h_\eta f}{p}{{\lambda\eta^{m-h}}},
	\end{align*}
	and applying \eqref{op sparse con norma p y pesos }, we conclude that 
	\vspace{-0.8cm}
	
	\begin{align*}	
		\normadeflpw{\mathcal{A}_{{\mathcal{S}}}^{m,h}(b,f)}{q}{\lambda}	
		&\lesssim \normadefp{b}{BMO^{\ex{\delta}}_\eta}^m	\normadeflpw{f}{p}{{\lambda\eta^{m}}}	
		=\normadefp{b}{BMO^{\ex{\delta}}_\eta}^m
		\normadeflpw{f}{p}{{\mu}}.
	\end{align*}		
\end{proof}

\begin{proof}[Proof of Theorem \ref{conmutador integral fraccionaria en lpv con simb bloom}]
	We first consider $b\in\lplocc{m}$.
	Since $\mu\in \lploc{p}$ and taking into account Theorem \ref{est puntual del conmutador de la int fracc de orden m}, it is enough to show that, for a sparse family $\mathcal{S}$, 
	\vspace{-0.3cm}
	
	$$\normadeflpw{\mathcal{A}_{{\mathcal{S},\alpha}}^{m,h}(b,f)}{q}{\lambda}\lesssim \normadeflpw{f}{p}{\mu}, \esp h\in\{0,1,...,m\},$$
	holds for each nonnegative bounded function with compact support $f$.
	Let $h\in\{0,1,...,m\}$, by duality
	\vspace{-0.6cm}
	
	\begin{align}\label{12}
	\normadeflpw{\mathcal{A}_{{\mathcal{S},\alpha}}^{m,h}(b,f)}{q}{\lambda}
	&\lesssim\sup_{\normadefpv{g}{q'}\leq1}	\int_{\z}\lambda(x)g(x)\mathcal{A}_{{\mathcal{S},\alpha}}^{m,h}(b,f)(x)\,dx \nonumber\\
	&=\sup_{\normadefpv{g}{q'}\leq1}
	\sum_{Q\in \mathcal{S}}|Q|^{\alpha/n}|(b-b_{Q})^{h}f|_{Q}\int_{Q}|b(x)-b_{Q}|^{m-h}\lambda(x)g(x)\,dx.
	\end{align}
	Let $\tilde{\mathcal{S}}$ be the sparse family provided by Proposition \ref{integral de b por f menor q la integral del operador sparse} and $Q\in\tilde{\mathcal{S}}$. Then, by this proposition and denoting $\eta=\nu^{1/m}$, we have that $\normadeflpw{\mathcal{A}_{{\mathcal{S},\alpha}}^{m,h}(b,f)}{q}{\lambda}$ is bounded by a multiple of
	\vspace{-0.4cm}
	 
	\begin{align*}
		\normadefp{b}{BMO_\eta^{{\ex{\delta}}}}^m
		\sup_{\normadefpv{g}{q'}\leq1}\sum_{Q\in \tilde{\mathcal{S}}} \frac{|Q|^{\alpha/n}\normadefpv{\ca}{n/\delta}^m}{|Q|}\int_{Q} (\mathcal{A}_{\tilde{\mathcal{S}}})^h_\eta f(x)\,dx
		\int_{Q} (\mathcal{A}_{\tilde{\mathcal{S}}})^{m-h}_\eta (\lambda g)(x)\,dx.
	\end{align*}
	Note that, if we denote $\ex{\beta}=n/(m\ex{\delta}+\alpha)$, by Lemma \ref{equivalenciabeta}, $|Q|^{\alpha/n}\normadefpv{\ca}{n/\delta}^m\simeq\normadefpv{\ca}{\beta}$. 
	Thus, we get the same inequality as in \eqref{misma estimacion}.
	So we can proceed in the same way as in the proof of Theorem \ref{teo simb lips bloom} to get the desired result for the case $b\in\lplocc{m}$.
	
	In order to complete the proof we must show that if $b\in BMO_{\eta}^{\ex{\delta}}$ then $b\in\lplocc{m}$.
	In fact, for any compact set $K$ we choose a cube $Q$ with $|Q|>1$ such that $K\subset Q$. Then
	\vspace{-0.4cm}
	
		\begin{align*}
			\int_K |b|^m
			\le \int_Q |b|^m
			\lesssim \int_Q |b-b_Q|^m + \left(\int_Q |b|\right)^m.
		\end{align*}
	Since $b\in\lplocc{1}$, the last term is bounded. For the first term note that, by Lemma \ref{integral de b por f menor q la integral del operador sparse},
	\vspace{-0.4cm}
	
		\begin{align*}
			\int_Q |b-b_Q|^m \ca
			&\lesssim  \normadefp{b}{BMO^{\ex{\delta}}_{\eta}}^m
			\normadefpv{\ca}{n/\delta}^{m}
			\int_{Q} (\mathcal{A}_{\tilde{\mathcal{S}}})^m_\eta \ca\\
			&\lesssim  \normadefp{b}{BMO^{\ex{\delta}}_{\eta}}^m
			\normadefpv{\ca}{n/\delta}^{m}
			\normadefpv{\lambda(\mathcal{A}_{\tilde{\mathcal{S}}})^m_\eta\ca}{q}
			\normadefpv{\lambda^{-1}\ca}{q'}.
		\end{align*}
	By Lemma \ref{factores en Apq implica factores en Apq} and \eqref{op sparse con norma q y pesos },
	\vspace{-0.4cm}
	
	\begin{align*}
		\normadefpv{\lambda(\mathcal{A}_{\tilde{\mathcal{S}}})^m_\eta\ca}{q}
		=\normadeflpw{(\mathcal{A}_{\tilde{\mathcal{S}}})^m_\eta\ca}{q}{\lambda}
		\lesssim 
		\normadefp{\ca}{L^{\ex{q}}_{\lambda\eta^m}}=
		\normadefp{\ca}{L^{\ex{q}}_\mu}.
	\end{align*}
	So we are done.
\end{proof}	
	

\end{document}